\documentclass[12pt]{amsart}

\usepackage{amsmath}
\usepackage{amssymb}  
\usepackage{latexsym} 
\usepackage{comment}
\usepackage{url}
\usepackage{fullpage,url,amssymb,amsmath,amsthm,amsfonts,mathrsfs}
\usepackage[usenames,dvipsnames]{color}
\usepackage[pagebackref = true, colorlinks = true, linkcolor = blue, citecolor = Green]{hyperref}
\usepackage[alphabetic,lite]{amsrefs}
\usepackage{enumitem}
\usepackage{amscd}   
\usepackage[all, cmtip]{xy} 
\usepackage{xfrac}
\usepackage[T1]{fontenc}

\usepackage[all]{xy}
\usepackage{fullpage}

\usepackage{color} 

\def\act#1#2%
  {\mathop{}%
   \mathopen{\vphantom{#2}}^{#1}%
   \kern-3\scriptspace%
   #2}

\newcommand{\Z}{{\mathbb Z}}

\newcommand{\Q}{{\mathbb Q}}

\newcommand{\C}{{\mathbb C}}
\newcommand{\F}{{\mathbb F}}

\newcommand{\PP}{{\mathbb P}}

\DeclareMathOperator{\Gal}{Gal}

\DeclareMathOperator{\Pic}{Pic}

\newcommand{\id}{\operatorname{id}}

\newtheorem{Theorem}{Theorem}[section]
\newtheorem{Lemma}[Theorem]{Lemma}

\newtheorem{Definition}[Theorem]{Definition}

\newtheorem{Remark}[Theorem]{Remark}

\numberwithin{equation}{section}

\newcommand{\Frob}{\operatorname{Frob}}

\begin{document}

\title{Recovering algebraic curves from L-functions of Hilbert class fields}

\author{Jeremy Booher}
\address{School of Mathematics and Statistics, University of Canterbury, Private Bag 4800, Christchurch 8140, New Zealand}
\email{jeremy.booher@canterbury.ac.nz}
\urladdr{http://www.math.canterbury.ac.nz/\~{}j.booher}

\author{Jos\'e Felipe Voloch}
\address{School of Mathematics and Statistics, University of Canterbury, Private Bag 4800, Christchurch 8140, New Zealand}
\email{felipe.voloch@canterbury.ac.nz}
\urladdr{http://www.math.canterbury.ac.nz/\~{}f.voloch}

\begin{abstract}
In this paper, we prove that a smooth {hyperbolic} projective
curve over a finite field can be recovered from L-functions associated to the Hilbert class field of the curve and its constant field extensions. As a consequence, we give a new proof of a result of Mochizuki and Tamagawa that two such curves with isomorphic fundamental groups are themselves isomorphic.
\end{abstract}

\maketitle

\section{Introduction}

It has long been known that the zeta function of a global field does not
determine the field uniquely (e.g. for function fields, take isogenous, but non-isomorphic, elliptic curves). 
Recently, there has been work done to 
recover a global field from more refined invariants of a similar nature.
For example, \cite{CDL+} proves that a global field can be recovered
from the collection of all its abelian L-functions. Another example is 
the conjecture \cite[Conjecture 2.2]{SV} which predicts that the zeta functions 
of the Hilbert class field $H(C)$ and successive iterates $H(H(C)),\ldots$ determines an
algebraic curve $C$ over a finite field up to
Frobenius twist. 
These two approaches are naturally related
to the classical work of Neukirch and Uchida of recovering global fields
from their absolute Galois groups \cite{Uchida}, and
the more recent work of Mochizuki and Tamagawa (see \cites{Moch,Tam}) of recovering algebraic curves
from their fundamental groups, respectively.

The purpose of this paper is twofold. First, we prove that a smooth proper curve of genus at least two (i.e. a proper hyperbolic curve) over a finite field can be recovered from L-functions associated to the Hilbert class field of the curve and its constant field extensions. 
Secondly, we show that two such curves with isomorphic fundamental groups are isomorphic.  This gives a new proof of the weak isomorphism version of the theorem of Mochizuki and Tamagawa mentioned earlier.
Our approach crucially depends on the work of Zilber \cites{Z1,Z2}, resolving a conjecture of Bogomolov, Korotiaev and Tschinkel \cite{BKT}.

\section{$L$-functions}

Let $q = p^a$ be a prime power, and $C$ be a smooth, projective, irreducible curve over $\F_q$ of genus at least one.  A divisor $D_1$ of degree one on $C$ gives an Abel-Jacobi embedding of $C$ into $J_C$ via $P \mapsto [P - D_1]$ on geometric points.  Throughout, we fix a choice of degree one divisor and Abel-Jacobi embedding for each curve; a degree-one divisor exists exists by \cite{Schmidt1931}.   

\begin{Definition}
Let $J_C$ be the Jacobian of $C$ and $\Phi: J_C \to J_C$ denote the $\F_q$-Frobenius map.  A Hilbert class field of $C$, with respect to a fixed Abel-Jacobi embedding of $C$ into $J_C$, is defined to be $H(C) := (\Phi-\id)^*(C) \subset J_C$.
\end{Definition}

The function field of $H(C)$ is a Hilbert class field of the function field of $C$ in the sense of class field theory  \cite[VIII.3]{AT}. Thus we retain the name even though we are considering the smooth projective curve throughout. The curve $H(C)$ is an unramified abelian cover of $C$ with
 Galois group $J_C(\F_q) =( \ker (\Phi - \id)(\F_q))$.
Note that changing $D_1$ twists $H(C)$.  Let $\mathcal{S}_C$ denote the set of places of $C$.

\begin{Definition}
Let $\chi : \Gal(H(C)/C) = J_C(\F_q) \to \C^*$ be a character.  For a place $P$ of $C$, let $\Frob_P \in \Gal(H(C)/C)$ denote the Frobenius at $P$. The $L$-series for $\chi$ is defined as
\[
L(t,C,\chi) := \prod_{P \in \mathcal{S}_C} ( 1 - \chi(\Frob_P) T^{\deg P})^{-1}.
\]
\end{Definition}

The following result is a special case of \cite[VI \S 5 Theorem 2]{Serre}.

\begin{Lemma}
For $P \in \mathcal{S}_C$, the Frobenius $\Frob_P$ is $[P - \deg(P) D_1] \in J_C(\F_q) = \Gal(H(C)/C)$, viewing $P$ as
a divisor of degree $\deg(P)$. 
\end{Lemma}

We see that
\begin{equation} \label{eq:lfunction}
L(t,C,\chi) =\sum_{D\ge0}\chi([D-\deg(D)D_1])t^{\deg D},
\end{equation}
\noindent
where the sum is over effective divisors on $C$ and $[.]$ associates a divisor of degree zero to its divisor class in
the Jacobian. {Note that this again depends on the fixed choice of $D_1$}.

We let $\Frob$ denote the $q$th power map.  We refer to the $q^m$-th power map $\Frob^m$ as the  $\F_{q^m}$-Frobenius, and refer to an arbitrary element of $\Gal(\bar{\F}_q / \F_q)$ as a generalized Frobenius.  These field automorphisms induce isomorphisms of schemes over $\bar{\F}_q$, and in particular on the base change to $\bar{\F}_q$ of schemes over $\F_q$.  If $C$ is a curve defined over $\F_q$, then the Frobenius twist $\Frob^m(C)$ and the curve $C$ are isomorphic as schemes (via the map induced by the field automorphism $\Frob^m$), but not as $\F_q$-schemes. 

Zilber \cites{Z1,Z2} resolved a conjecture of Bogomolov, Korotiaev and Tschinkel \cite{BKT}: 

\begin{Theorem}[{\hspace{1pt}\cite[Theorem 1.3]{Z2}}] \label{thm:zilber}
Let $C$ and $C'$ be smooth, projective, irreducible curves over $\F_q$ of genus at least two.  If  $\psi : J_C(\bar{\F}_q) \to J_{C'}(\bar{\F}_q)$ is a group isomorphism such that $\psi(C(\bar{\F}_q)) = C'(\bar{\F}_q)$ (with respect to the fixed Abel-Jacobi embeddings), then $\psi$ arises from a morphism of curves composed with a limit of Frobenius maps. 

 More precisely, there exists an integer $m$, an isomorphism $\alpha: \Frob^m(C) \simeq C'$, and a generalized Frobenius $\beta: \bar{\F}_q \to \bar{\F}_q$ restricting to $\Frob^{-m}$ on $\F_q$ such that $\psi$ is the map on Jacobians induced by $\alpha \circ \beta : C \to C'$. 
\end{Theorem}


We use this to prove our first main result.

\begin{Theorem}
\label{main}
Let $C,C'$ be smooth projective curves of genus at least two
over a finite field $\F_q$. 
Suppose $\psi : J_C(\bar{\F}_q) \to J_{C'}(\bar{\F}_q)$ is a set-theoretic map
inducing an isomorphism of groups between $J_C(\F_{q^n})$ and $J_{C'}(\F_{q^n})$ for every $n \ge 1$.
If $$L(t, C \otimes \F_{q^n}, \chi) = L(t, C' \otimes \F_{q^n}, \chi \circ \psi|_{J_C(\F_{q^n})})$$ for all $n$ and all characters $\chi$ of $J_C(\F_{q^n})$, then $C$ and $C'$ are Frobenius twists of each other.  



\end{Theorem}

\begin{proof} 
We show that $\psi(C(\F_{q^n})) = C'(\F_{q^n})$ for all $n \geq 1$ 
under our hypotheses, in order to apply Zilber's theorem.
It suffices to treat the case $n=1$ as the same arguments applied to $C \otimes \F_{q^n}$ and $C' \otimes \F_{q^{n}}$ yield the general result. Using \eqref{eq:lfunction} we write
$$L(t,C,\chi) = 
\sum_{x \in J_C(\F_q)} \chi(x)\sum_{d=0}^{\infty} \#\{D\ge 0: D\sim x +dD_1\} t^d.$$

Knowledge of all these $L$-functions for all characters $\chi$ give us
the values of $\#\{D\ge 0: D\sim x +dD_1\}$ for each $x \in J_C(\F_q), d \ge0$. Finally $x + D_1 \sim P, P \in C(\F_q)$ if and only if 
$\#\{D\ge 0: D\sim x +D_1\} = 1$. Indeed, if $x + D_1$ is not linearly
equivalent to a point then $\#\{D\ge 0: D\sim x +D_1\} = 0$ and if it is, then
$\#\{D\ge 0: D\sim x +D_1\} = 1$ since the curve has positive genus so
distinct points are not linearly equivalent.  This shows that $\psi(C(\F_{q^n})) = C'(\F_{q^n})$; applying Theorem~\ref{thm:zilber} completes the proof. 
\end{proof}

\begin{Remark}
The proof actually gives more.  The proof only uses that the derivatives of $L(t,C \otimes \F_{q^n},\chi)$ and of $L(t, C' \otimes \F_{q^n}, \chi \circ \psi|_{J_C(\F_{q^n})})$ at $t=0$ are equal for all $n$ and $\chi$. 
Secondly, the proof shows there is {an integer $m$,} an isomorphism $\alpha : \Frob^m(C) \to
C'$, and a generalized Frobenius $\beta : \bar{\F}_q \to \bar{\F}_q$
restricting to $\Frob^{-m}$ on $\F_q$ such that $\psi$ is the composition
\[
J_C(\bar{\F}_q) \overset{\beta} \to J_{\Frob^m(C)} (\bar{\F}_q)
\overset{\alpha} \to J_{C'}(\bar{\F}_q).
\]
\end{Remark}

\begin{Remark}
The theorem fails for genus one. In this case, for any non-trivial character $\chi$ we have $L(t,C\otimes\F_{q^n},\chi) =1$, and for
the trivial character one gets the zeta function of $C\otimes\F_{q^n}$, which only detects the isogeny class of $C$. If we have 
two elliptic curves with isomorphic endomorphism rings (which exist if the class number of the ring is at least two), then
results of Kohel \cite[Propositions 21 and 22]{Kohel} imply the existence of two isogenies
of coprime orders between the elliptic curves (which are their own Jacobians) and, in particular we have an isomorphism of groups between $C(\F_{q^n})$ and $C'(\F_{q^n})$ for every $n \ge 1$ inducing equality of all corresponding $L$-functions.
\end{Remark}

\begin{Remark}
It is not generally enough to have $J_C(\F_{q})$ and $J_{C'}(\F_{q})$ isomorphic and the corresponding $L$-functions matching (i.e.
the conditions only for $n=1$). To see this, let $C/\F_3$ be a curve of genus two with non-constant map of degree two $\pi :C \to \PP^1$ such that:
\begin{itemize}
\item  $C$ has no rational or quadratic Weierstrass points,
\item  only one rational point of $\PP^1$ splits in $C$ under $\pi$,
\item  $J_C(\F_{3})$ has order $5$.
\end{itemize}
\cite[Example 4.1]{SV} provides
two non-isomorphic curves satisfying these conditions.  But these conditions will determine $L(t,C,\chi)$ for each character $\chi : J_C(\F_3) \to \C^*$, providing the desired example.

The conditions imply one fiber of $\pi$ consists of rational points and the other three fibers of $\pi$
consist of pairs of quadratic points. Let $C(\F_3) = \{P,Q\}$. The fibers give linearly equivalent effective divisors of degree two. The divisors $2P, 2Q$ are also effective divisors of degree two but are not equivalent to the fibers or to each other. As 
$J_C(\F_{3})$ has order $5$ there are two more classes of divisors of degree two, each represented by a single effective
divisor consisting of conjugate quadratic points. So $\#C(\F_3) = 2, \#C(\F_9) = 12$ and the zeta function of $C$ (which is
the $L$-function with trivial character) is determined. 

We use $D_1 = P$ as the divisor of degree $1$ embedding $C$ in $J_C$.  For divisors of degree one, $P \mapsto 0$ and $Q \mapsto Q-P$; notice the latter generates $J_C(\F_{3})$.
For divisors of degree $2$, we have $D \mapsto D -2P$.  Thus $P+Q \mapsto Q-P$ and likewise for the other fibers of
$\pi$ as they are linearly equivalent.
 Furthermore $2P \mapsto 0$ and $2Q \mapsto 2(Q-P)$.  The two other effective divisors of degree $2$ map to
$3(Q-P),4(Q-P)$ respectively. Since we know that $L(t,C,\chi)$ is a quadratic polynomial, if $\chi \ne 1$ this is enough information to compute it.  In particular, if $\chi(Q-P) = \zeta$ then using \eqref{eq:lfunction} we have
$$L(t,C,\chi) = 1 + (1+\zeta)t + (4\zeta +1+ \zeta^2+ \zeta^3 + \zeta^4)t^2 =  1 + (1+\zeta)t + 3\zeta t^2.$$

\end{Remark}

\begin{Remark}
All of the $L$-functions appearing in the theorem can be realized as (non-abelian)
$L$-functions over the base field {up to a harmless change of variable}. Recall that for a Galois extension of global fields $L/K$ with Galois group $G$, a subgroup $H \subset G$ with $K' = L^H$, and a representation $\rho$ of $H$, we have that (note the change of variable
to $s$, such that $t=q^{-s}$ in the function field case).
\[
L(s,L/K',\rho) = L(s, L/K,\operatorname{Ind}^G_H(\rho)).
\]
Let $K = \F_q(C)$ and $L$ be the function field of the cover of $C$ corresponding to the cover of $J_C$ given by $(\Frob^n - \id)^*$; note this cover includes an extension of the constant field when $n>1$.   We see that for a character $\chi$ of $J_C(\F_{q^n})$, 
\[
L(t^n,C \otimes {\F_{q^n}}, \chi) = L(t,C, \operatorname{Ind}^{\Gal(L/K)}_{J_C(\F_{q^n})} (\chi)).
\]
{The equality of $L$-functions in Theorem~\ref{main} can be checked after the change of variable $t \mapsto t^n$, so can be checked using L-functions over the base field.}
\end{Remark}

\section{Fundamental groups}

Recall the fundamental exact sequence of \'etale fundamental groups

\begin{equation}
\label{fund}
1 \rightarrow \pi_1(\bar{C}) \rightarrow \pi_1(C) \rightarrow G_{\F_q} \rightarrow 1,
\end{equation}
where $\bar{C} = C\otimes\bar{\F}_q$ and $G_{\F_q}$ is the absolute Galois group of $\F_q$.  We use the notation $p_C: \pi_1(C) \rightarrow G_{\F_q}$ for the right map in the above sequence.
 
As mentioned in the introduction, we give a new proof of the following theorem.
 
\begin{Theorem}[Mochizuki--Tamagawa]
\label{moch}
 Let $C,C'$ be smooth projective curves of genus at least two
over a finite field $\F_q$. 
If there is an isomorphism $\pi_1(C) \simeq \pi_1(C')$ of groups 
then $C$ is isomorphic to $C'$ as schemes.
\end{Theorem}

\begin{proof}
First, it is already known that $\pi_1(C)$ group-theoretically determines the Frobenius element in $G_{\F_q}$ \cite[Proposition 3.4]{Tam}. Thus 
 the isomorphism $\pi_1(C) \simeq \pi_1(C')$ induces the identity on $G_{\F_q}$ via $p_C$ and $p_{C'}$ and induces isomorphisms $\pi_1(C \otimes \F_{q^n}) \simeq \pi_1(C' \otimes \F_{q^n})$ for each $n \geq 1$.

Next, for $n \geq 1$ there is a short exact sequence 
\begin{equation}
0 \to J_C(\F_{q^n}) \to \widehat{\Pic_C(\F_{q^n})} \to \widehat{\Z} \to 0.
\end{equation}
obtained by taking the profinite completion of the short exact sequence arising from the degree map $\Pic_C(\F_{q^n}) \to \Z$.
Class field theory identifies $\widehat{\Pic_C(\F_{q^n})}$ with $\pi_1(C \otimes \F_{q^n})^{\rm{ab}}$ and the degree map with $p_{C \otimes \F_{q^n}}$.   It follows that $\pi_1(C \otimes \F_{q^n})^{\rm{ab,tor}}$ is identified with $J_C(\F_{q^n})$. 
The (group-theoretic) transfer map $\pi_1(C)^{\rm{ab}} \to \pi_1(C \otimes \F_{q^n})^{\rm{ab}}$ is compatible with the Artin map, and induces the inclusion map $J_C(\F_q) \to J_C(\F_{q^n})$ \cite[VII.8]{serrelocal}.
Therefore the isomorphism $\pi_1(C) \simeq \pi_1(C')$ determines
isomorphisms $\pi_1(C \otimes \F_{q^n}) \simeq \pi_1(C' \otimes \F_{q^n})$ and hence
 isomorphisms $\psi_n : J_C(\F_{q^n}) \simeq J_{C'}(\F_{q^n})$.  For $m \geq 1$, $\psi_n$ and $\psi_{nm}$ are compatible with the inclusion maps $J_C(\F_{q^n}) \to J_C(\F_{q^{nm}})$ and $J_{C'}(\F_{q^n}) \to J_{C'}(\F_{q^{nm}})$, and hence we obtain a bijection $\psi : J_C(\overline{\F}_q) \to J_{C'}(\overline{\F}_q)$ as in Theorem~\ref{main}.


Now let $J_X$ be the Jacobian of a curve $X$ defined over a finite field $k$ with fixed Abel-Jacobi embedding.  Analogously to \eqref{fund}, there is a fundamental exact sequence 
\begin{equation}
1 \rightarrow \pi_1(\overline{J_X}) \rightarrow \pi_1(J_X) \rightarrow G_{k} \rightarrow 1.
\end{equation}
Furthermore $\pi_1(\overline{J_X})$ is the abelianisation of $\pi_1(\bar{X})$.  Thus
the prime-to-$p$ part 
of $\pi_1(\overline{J_X})$ is the product of the Tate modules of $J_X$.
The fundamental exact sequence describes the Galois action on the
Tate module, and so the Tate module with its Galois action is determined by $\pi_1(\overline{J_X})$ (equivalently by $\pi_1(\bar{X})^{\rm{ab}}$ using the fixed Abel-Jacobi embedding). 

We apply this with $k = \F_{q^n}$ and $X := X_n = H( C \otimes \F_{q^n})$.  Let $\Phi$ denote the $\F_q$-Frobenius on $J_C$.
For characters $\chi$ of $J_C(\F_{q^n})$, the $L(t,C\otimes\F_{q^n},\chi)$ are the characteristic polynomials of $\Phi^n$ in the
$\chi$-eigencomponent of the Tate module of $J_{X_n}$.  Notice that $\pi_1(X_n)$ is determined group-theoretically by $\pi_1(C)$ as it is the kernel of the morphism $\pi_1(C\otimes\F_{q^n}) \to \pi_1(C \otimes \F_{q^n})^{\rm{ab,tor}} = J_C (\F_{q^n}) $ and we already know that $\pi_1(C\otimes\F_{q^n})$ is determined group-theoretically.
Then the previous paragraph applies, showing that the Tate modules and hence $L$-functions for $C$ and $C'$ agree.  
Theorem \ref{main} shows that $C$ and $C'$ are Frobenius twists, and hence are isomorphic as schemes. 
\end{proof}

\begin{Remark}
Theorem~\ref{moch} shows how to recover projective curves from their fundamental groups.   
It is known to experts that the affine case follows relatively easily.
Historically the argument went in the other direction; Mochizuki proved the projective case building on work of Tamagawa which treated affine curves \cites{Moch,Tam}. 
\end{Remark}

\begin{Remark}
We do not need the full data of an isomorphism $\pi_1(C) \simeq \pi_1(C')$ in order to apply Theorem~\ref{main}.  In particular, we only need two pieces of information for each $n \geq 1$:
\begin{itemize}
\item an isomorphsim $\pi_1(C \otimes \F_{q^n})^{\rm{ab}} \simeq \pi_1(C' \otimes \F_{q^n})^{\rm{ab}}$;
\item isomorphisms $\pi_1(X_{n,\bar{\F}_q}) ^{\rm{ab}} \simeq \pi_1(X'_{n,\bar{\F}_q})^{\rm{ab}}$ that are compatible with the action by the Frobenius and actions by $J_C(\F_{q^n})$ and $J_{C'}(\F_{q^n})$. 
\end{itemize}
  The first identifies $J_C(\F_{q^n}) \simeq J_{C'}(\F_{q^n})$ and similarly for geometric points.  The latter isomorphisms, or even just an isomorphism of the abelian pro-$\ell$ fundamental groups for any prime $\ell \neq p$, suffices to identify the $L$-functions.  In particular, it identifies $H^1(X_{n, \bar{\F}_q}, \Q_\ell) \simeq H^1(X'_{n, \bar{\F}_q}, \Q_\ell)$ with an action of $J_C(\F_{q^n}) \simeq \Gal(X_n / C \otimes \F_{q^n})$ and of Frobenius.  The $L$-function is the characteristic polynomial of the $\F_{q^n}$-Frobenius on the $\chi$-component of the cohomology $H^*(X_{n, \bar{\F}_q}, \Q_\ell)$.   
   In particular, an isomorphism between the ``geometrically abelian by abelian pro-$\ell$'' fundamental groups suffices. 
The required actions on $\pi_1(X_{n,\bar{\F}_q}) ^{\rm{ab}}$ are encoded in this group via conjugation by lifts of Frobenius and by lifts of elements of $J_C(\F_{q^n})$ through $\pi_1(C \otimes \F_{q^n}) \to \pi_1(C \otimes \F_{q^n})^{\rm{ab}}$.
   \end{Remark}


\begin{Remark}
Jakob Stix suggested an alternative way to deduce Theorem \ref{moch} from Zilber's theorem without going through
Theorem \ref{main}. As in the proof of Theorem \ref{moch}, construct the map $\psi$ between the points of 
the respective Jacobians. To prove that $\psi$ induces a bijection between the points of the respective curves, 
Stix uses Tamagawa's \cite[Proposition 0.7]{Tam} characterization of those sections of the fundamental exact sequence that come
from rational points on the curve. He shows that the composition of a section $G_{\F_q} \rightarrow \pi_1(C)$ coming
from a point with the isomorphism $\pi_1(C) \to \pi_1(C')$, yields a section $G_{\F_q} \rightarrow \pi_1(C')$ also coming
from a point and similarly over $\F_{q^n}$. Moreover, the resulting map $C(\overline{\F}_q) \to C'(\overline{\F}_q)$ is
compatible with the map $J_C(\overline{\F}_q) \to J_{C'}(\overline{\F}_q)$, at which point Zilber's theorem applies.
\end{Remark}

\section*{Acknowledgements}
The authors were supported by the Marsden Fund Council administered by the Royal Society of New Zealand. They would also like to thank Jakob Stix {and the referees} for many helpful comments.


\end{document}